\documentclass{amsart}
\usepackage{amssymb}
\usepackage{amscd}
\usepackage[latin1]{inputenc}
\usepackage{amsmath}
\usepackage{amsfonts}
\usepackage{latexsym}
\usepackage{verbatim}
\usepackage{graphicx}
\usepackage{epsfig}
\newtheorem{theorem}{Theorem}[section]
\newtheorem{corollary}[theorem]{Corollary}
\newtheorem{lemma}[theorem]{Lemma}

\newtheorem{proposition}[theorem]{Proposition}

\theoremstyle{definition}
\newtheorem{definition}[theorem]{Definition}
\newtheorem{example}[theorem]{Example}
\newtheorem{remark}[theorem]{Remark}

\newcommand{\zz}{\mathbb{Z}}

\begin{document}

\title[On entropy and intrinsic ergodicity of coded subshifts]{On entropy and intrinsic ergodicity of coded subshifts}

\begin{abstract}
Any coded subshift $X_C$ defined by a set $C$ of code words contains a subshift, which we call $L_C$, consisting of limits of single code words. We show that when $C$ satisfies a unique decomposition property, the topological entropy $h(X_C)$ of $X_C$ is determined completely by $h(L_C)$ and the number of code words of each length. More specifically, we show that $h(X_C) = h(L_C)$ exactly when a certain infinite series is less than or equal to $1$, and when that series is greater than $1$, we give a formula for $h(X_C)$. In the latter case, an immediate corollary 
(using a result from \cite{CT0}) is that $X_C$ has a unique measure of maximal entropy. 
\end{abstract}

\date{}
\author{Ronnie Pavlov}
\address{Ronnie Pavlov\\
Department of Mathematics\\
University of Denver\\
2390 S. York St.\\
Denver, CO 80208}
\email{rpavlov@du.edu}
\urladdr{www.math.du.edu/$\sim$rpavlov/}
\thanks{The author gratefully acknowledges the support of NSF grant DMS-1500685.}
\renewcommand{\subjclassname}{MSC 2010}
%below are the definitions for some subject classification numbers
%22D40 Ergodic theory on groups 
%37A05 Measure-preserving transformations 
%37A15 General groups of measure preserving transformations 
%37A35 Entropy and other invariants, isomorphism, classification 
%37B10 symbolic dynamics
%37B40 topological entropy
%37B50 multi-dimensional shifts of finite type, tiling systems 
%37C40 smooth ergodic theory, invariant measures
%37C45 dimension theory of dynamical systems
%37C85 Dynamics of group actions other than Z and R, and foliations 
%37C99 smooth dynamical systems, general theory
%37D35 thermodynamic formalism, variational principles, equilibrium states
\maketitle

%For example, it was shown in (ME, VAUGHN/ME) that certain weakened specification properties do not necessarily imply uniqueness of the measure of maximal entropy via coded subshift examples. The arguments in those works were somewhat ad hoc, and amounted to showing that a coded subshift  and

\section{Introduction}\label{intro}

The class of coded subshifts has been a fruitful area of research in symbolic dynamics, both in terms of general properties (\cite{codedbook}, \cite{blanhan}, \cite{kwietcoded}, \cite{fiebigs}, \cite{kriegercoded}) and as a source of examples with interesting properties. Informally, a coded subshift is defined as the set of all limits of arbitrary bi-infinite concatenations of finite \textbf{code words} from a predetermined set; see Section~\ref{defs} for a formal definition. Recently, several works (\cite{vaughntowers}, \cite{vaughnronnie}, \cite{pavlovspec}) have examined whether various coded subshifts have unique measure of maximal entropy. In those works, a sort of dichotomy seems to arise for coded subshifts, coming from whether the act of concatenation increases topological entropy or not. (For instance, if the set of code words is all words on $\{0,1\}$, then nothing is gained by the step of allowing arbitrary concatenations.)

The main motivation for this work is to formalize this dichotomy and show that a single infinite series controls much of the behavior of any coded subshift. The series in question is motivated by a natural connection between coded subshifts and countable-state topological Markov chains, and in fact the main purpose of this work is to formalize this connection in a way that has not (to our knowledge) yet been done.

We need to make some definitions to state our main results. Say that $C \subseteq A^*$ is an arbitrary set of words over a finite alphabet $A$, and for each $n$ define $C_n = C \cap A^n$. Define $B_C$ to be the set of all biinfinite concatenations of the words in $C$, i.e. 
\[
B_C = \{x \in A^{\mathbb{Z}} \ : \ \exists s_k \rightarrow \infty \textrm{ s.t. } \forall k, x([s_k, s_{k+1})) \in C\}.
\]
We say that $C$ has \textbf{unique decipherability} if every $x \in B_C$ is associated to a single such sequence $s_k$. A weaker property is that of \textbf{unique decomposition}, which means that no finite word can be written as a concatenation of words in $C$ in two different ways. 
The set $B_C$ is shift-invariant, but may not be closed; therefore we define the \textbf{coded subshift} associated to $C$ to be $X_C = \overline{B_C}$. We also define the following important subset of $X_C$:
\[
L_C = \{x \in A^{\mathbb{Z}} \ : \ \forall k, x([-k,k]) \textrm{ is a subword of some word in } C\}.
\]
Clearly $L_C \subseteq X_C$. We note that $L_C$ is empty if and only if $C$ is finite.
Most behavior of the coded subshift $X_C$ is given by $L_C$ and $B_C$, since every invariant measure $\mu$ on $X_C$ has $\mu(L_C \cup B_C) = 1$; see Lemma~\ref{LB}.

As mentioned earlier, various recent works have suggested that much of the behavior of $X_C$ is determined by whether the inequality $h(X_C) > h(L_C)$ holds. Most notably, it was shown in \cite{CT0} (see Theorem\ref{CTthm} below) that this inequality implies uniqueness of the measure of maximal entropy, and in \cite{vaughntowers} (see Theorem~\ref{vaughnthm} below) that under the additional assumption of unique decipherability, that measure has several desirable statistical properties (for instance, it is essentially Bernoulli, modulo possible periodic behavior). In the other direction, the examples from \cite{vaughnronnie} and \cite{pavlovspec} with multiple measures of maximal entropy are obtained by choosing $C$ so that $L_C$ is a simple subshift with multiple MMEs, and then proving that $h(X_C) = h(L_C)$.

The main results of this work show that the generating function
\[
f_C(\alpha) = \sum_{j \in \mathbb{N}} |C_j| e^{-j\alpha}
\]
completely determines whether $h(X_C) > h(L_C)$, and can be used to solve for $h(X_C)$ when it is not equal to $h(L_C)$. We note that $f_C(\alpha)$ may be infinite, but that it is continuous and strictly decreasing (since at least one $|C_j|$ is positive) on its interval of convergence. 

Our main results are the following.

\begin{theorem}\label{mainthm2}
If $f_C(h(L_C)) < 1$, then $h(X_C) = h(L_C)$ and every measure of maximal entropy on $X$ has support contained in $L$.
\end{theorem}

\begin{theorem}\label{mainthm1}
If $f_C(h(L_C)) = 1$, then $h(X_C) = h(L_C)$, and there may or may not be a measure of maximal entropy on $X$ with support not contained in $L$.
\end{theorem}

\begin{theorem}\label{mainthm3}
If $f_C(h(L_C)) > 1$ and $C$ has unique decomposition, then $h(X_C) > h(L_C)$ and in fact $h(X_C)$ is the unique solution to the equation 
$f_C(x) = 1$.
\end{theorem}

As mentioned earlier, it is known that $h(X_C) > h(L_C)$ implies some useful properties for $X_C$. (In both of the following results, we give versions using our notation which are equivalent to those in the referenced works by Corollary~\ref{subwords}.)

\begin{theorem}\label{CTthm}{\rm (\cite{CT0}, Theorem B)}
If $X_C$ is a coded subshift, and if $h(X_C) > h(L_C)$, then $X_C$ has a unique measure of maximal entropy.
\end{theorem}

Under the additional assumption of unique decipherability, \cite{vaughntowers} yields more information. In fact, the results there are more general, applying to the equilibrium state of any 
H\"{o}lder continuous potential $\phi$. We state a version here only for MMEs (which correspond to equilibrium states for $\phi = 0$).

\begin{theorem}\label{vaughnthm}{\rm (\cite{vaughntowers}, Theorem 1.6)}
If $X_C$ is a coded subshift on a finite alphabet generated by a set $C$ of code words with unique decipherability, and if $h(X_C) > h(L_C)$, then the unique measure of maximal entropy on $X_C$ satisfies additional conditions (ii) - (iv) from Theorem 1.1 of \cite{vaughntowers}.
\end{theorem}

The following corollary is immediate.

\begin{corollary}\label{maincor}
If $f_C(h(L_C)) > 1$ and $C$ has unique decomposition, then $X_C$ has a unique measure of maximal entropy (which satisfies additional conditions (ii) - (iv) from Theorem 1.1 of \cite{vaughntowers} if $C$ has unique decipherability).
\end{corollary}

Theorems~\ref{mainthm2}, \ref{mainthm1}, and \ref{mainthm3} are proved via the following auxiliary results, which may be of independent interest.

\begin{theorem}\label{aux1}
If $\alpha > h(L_C)$ and $f_C(\alpha) < 1$, then $h(X_C) \leq \alpha$.
\end{theorem}

\begin{theorem}\label{aux2}
If $f_C(\alpha) > 1$ and $C$ has unique decomposition, then $h(X_C) > \alpha$.
\end{theorem}

\section*{acknowledgments} 
The author would like to thank Karl Petersen and Omri Sarig for several useful discussions.

\section{General definitions}\label{defs}

\begin{definition}
For any finite alphabet $A$, the \textbf{full shift} over $A$ is the set $A^{\zz} = \{\ldots x_{-1} x_0 x_1 \ldots \ : \ x_i \in A\}$, which is viewed as a compact topological space with the (discrete) product topology.
\end{definition}

\begin{definition}
A \textbf{word} over $A$ is a member of $A^{\{i,i+1,\ldots,j\}}$ for some $i<j$, whose \textbf{length} $j-i+1$ is denoted by $|w|$. The set $\bigcup_{i,j \in \zz, i<j} A^{\{i,i+1,\ldots,j\}}$ of all words over $A$ is denoted by $A^*$. For any $n$, we use $A^n$ to denote the set $A^{\{1,\ldots,n\}}$.
\end{definition}

\begin{definition}
The \textbf{shift action}, denoted by $\{\sigma^n\}_{tn \in \zz}$, is the $\zz$-action on a full shift $A^{\zz}$ defined by $(\sigma^n x)_m = x_{m+n}$ for $m,n \in \zz$. 
\end{definition}

\begin{definition}
A \textbf{subshift} is a closed subset of a full shift $A^{\zz}$ which is invariant under the shift action, which is a compact space with the induced topology from $A^{\zz}$.
\end{definition}

The single shift $\sigma := \sigma^1$ is an automorphism on any subshift, and so for any subshift $X$, $(X,\sigma)$ is a topological dynamical system. 

\begin{definition}
The \textbf{language} of a subshift $X$, denoted by $\mathcal{L}(X)$, is the set of all words which appear in points of $X$. For any $n \in \zz$, $\mathcal{L}_n(X) := \mathcal{L}(X) \cap A^n$, the set of words in the language of $X$ with length $n$. 
\end{definition}

%In the previous definition, we dealt only with words from $A^n$ rather than $A^{\{i,\ldots,j\}}$ for arbitrary $i < j$; this is because any word in $A^{\{i,\ldots,j\}}$ can clearly be thought of as a word in $A^{j-i+1}$ by simply shifting it. We will generally consider two words to be the same if they are shifts of each other. 

\begin{definition}
For any subshift and word $w \in \mathcal{L}_n(X)$, the \textbf{cylinder set} $[w]$ is the set of all $x \in X$ with $x_1 x_2 \ldots x_n = w$.  
\end{definition}

The main class of subshifts which we treat in this work are the coded subshifts. 

\begin{definition}
For any set $C$ of words on an alphabet $A$, define 
\[
B_C = \{x \in A^{\mathbb{Z}} \ : \ \exists s_k \rightarrow \infty \textrm{ s.t. } \forall k, x([s_k, s_{k+1})) \in C\},
\]
\[  
L_C = \{x \in A^{\mathbb{Z}} \ : \ \forall k, x([-k,k]) \textrm{ is a subword of some word in } C\},
\]
and let $X_C = \overline{B_C}$. We call $X_C$ the \textbf{coded subshift} associated to $C$.
\end{definition}

As mentioned in the introduction, to any $C$ we also associate the following generating function:
\[
f_C(\alpha) = \sum_{j \in \mathbb{N}} |C_j| e^{-j\alpha}.
\]

\begin{definition}
A set $C$ of words on an alphabet $A$ has \textbf{unique decipherability} if no bi-infinite sequence can be written as a bi-infinite concatenation of words in $C$ in multiple ways.
\end{definition} 

\begin{definition}
A set $C$ of words on an alphabet $A$ has \textbf{unique decomposition} if no word can be written as a finite concatenation of words in $C$ in multiple ways.
\end{definition}

\begin{definition}\label{topent}
The \textbf{topological entropy} of a subshift $X$ is
\[
h(X) := \lim_{n \rightarrow \infty} \frac{1}{n} \ln |\mathcal{L}_n(X)|.
\]
\end{definition}

A standard subadditivity argument shows that this limit is in fact an infimum, i.e. for all $n$, 
$h(X) \leq \frac{1}{n} \ln |\mathcal{L}_n(X)|$. Alternately, for all $n$,
\begin{equation}\label{wordcount}
|\mathcal{L}_n(X)| \geq e^{nh(X)}.
\end{equation}

We also need some definitions from measure-theoretic dynamics; all measures considered in this paper will be Borel probability measures on a full shift $A^{\mathbb{Z}}$.

\begin{definition}
A measure $\mu$ on $A^{\mathbb{Z}}$ is {\bf ergodic} if any measurable set $C$ which is shift-invariant, meaning $\mu(C \triangle \sigma C) = 0$, has measure $0$ or $1$. 
\end{definition}

Not all $\sigma$-invariant measures are ergodic, but a well-known result called the ergodic decomposition shows that any non-ergodic measure can be written as a ``convex combination'' (formally, an integral) of ergodic measures; see Chapter 6 of \cite{walters} for more information. %One application of ergodic measures comes from Birkhoff's pointwise ergodic theorem, stated here only for the case of ergodic $\mu$ on a full shift $A^{\zz}$.

%\begin{theorem}{\rm (Birkhoff's pointwise ergodic theorem)}\label{birkhoff}
%For any ergodic measure $\mu$ on a subshift $X$ and any $f \in L^1(A^{\zz},\mu)$,
%\[
%\lim_{n \rightarrow \infty} \frac{1}{2n+1} \sum_{i=-n}^n f(\sigma^i x) \underset{\mu{\rm -a.e.}}{\rightarrow} \int f \ d\mu.
%\]
%\end{theorem}

\begin{definition}\label{measent}
For any $\sigma$-invariant measure $\mu$ on a full shift $A^{\zz}$, the \textbf{measure-theoretic entropy} of $\mu$ is
\[
h(\mu) := \lim_{n \rightarrow \infty} \frac{-1}{n}  \sum_{w \in A^n} \mu([w]) \ln \mu([w]),
\]
where terms with $\mu([w]) = 0$ are omitted from the sum.
\end{definition}

%We also need the following fact, which is essentially just an application of Birkhoff's ergodic theorem; a proof (of a more general version) can be found as Lemma 4.8 in \cite{pavlovperturb}. 

%\begin{lemma}\label{genericcount}
%For any subshift $X$, any ergodic measure $\mu$ on $X$, any word $w \in \mathcal{L}_{n}(X)$, any $n \in \mathbb{N}$, and any $\epsilon > 0$, define the set $C_{n,\epsilon,w}(X)$ to be the set of all $w \in \mathcal{L}_n(X)$ which have between $n(\mu([w]) - \epsilon)$ and $n(\mu([w]) + \epsilon)$ occurrences of $w$. Then,
%\[
%\liminf_{n \rightarrow \infty} \frac{\ln |C_{n,\epsilon,w}(X)|}{n} \geq h(\mu).
%\] 
%\end{lemma}

\begin{definition}
For any subshift $X$, a \textbf{measure of maximal entropy} on $X$ is a measure $\mu$ with support contained in $X$ for which $h(\mu) = h(X)$.
\end{definition}

It is well-known that every subshift has at least one measure of maximal entropy, and the ergodic decomposition and affineness of the entropy map (see Theorem 8.7(ii) in \cite{walters}) imply that every measure of maximal entropy on a subshift is a convex combination of ergodic measures of maximal entropy. %For full shifts in particular, it is well-known that there is only one measure of maximal entropy, namely the uniform Bernoulli measure $\mu$ defined by $\mu([w]) = |A|^{-n}$ for all $n$ and $w \in A^n$.

\section{Countable-state topological Markov chains}\label{markov}

A \textbf{countable-state topological Markov chain} is given by a countable graph $G$, which for our purposes we assume to be connected; the associated edge shift $E(G)$ is just the set of biinfinite paths on $G$, where all edges are considered distinct objects. There is a well-known classification of $E(G)$, for which we need some notation. For any vertex $u$, define $p(u,n)$ to be the number of cycles of length $n$ which begin and end with $u$, and define $s(u,n)$ to be the number of those cycles (of length $n$) whose only occurrences of $u$ are at the beginning and end. The \textbf{Gurevich entropy} $h(G)$ is then defined by
\[
h(G) = \lim_{n \rightarrow \infty} \frac{\ln |p(u,n)|}{n}.
\]
Connected countable graphs $G$ can then be classified into three categories, according to the following classification of Vere-Jones (originally proved in \cite{verejones}; also see \cite{ruette}).\\

\noindent
$\bullet$ $G$ is \textbf{transient} if $\sum_{n = 1}^{\infty} s(u,n) e^{-nh(G)} < 1$.

\noindent
$\bullet$ $G$ is \textbf{null recurrent} if $\sum_{n=1}^{\infty} s(u,n) e^{-nh(G)} = 1$, $\sum_{n = 1}^{\infty} s(u,n) ne^{-nh(G)} = \infty$.

\noindent
$\bullet$ $G$ is \textbf{positive recurrent} if $\sum_{n=1}^{\infty} s(u,n) e^{-nh(G)} = 1$, $\sum_{n = 1}^{\infty} s(u,n) ne^{-nh(G)} < \infty$.\\

The following theorem will be relevant for our purposes.

\begin{theorem}{\rm (\cite{gurevic3})}
There exists a Borel probability measure on $E(G)$ with entropy $h(G)$ if and only if $G$ is positive recurrent.
\end{theorem}

We now return to coded subshifts. Given a set $\mathcal{C}$, there is a natural countable labeled graph $\mathcal{G}_C = (G_C,\ell_C)$ associated to it, which is obtained by fixing a distinguished vertex $u$ and, for every word $w$ in any $C_n$, associating a loop of length $n$ labeled by $w$ which starts and ends at $u$; all of these loops are vertex-disjoint except for $u$. We note that in $G_C$, $s(u,n) = |C_n|$ and $p(u,n) \leq |\mathcal{L}_n(X)|$, and so $h(X_C) \geq h(G_C)$. We also note that $f_C(h(G_C)) = \sum_{n=1}^{\infty} s(u,n) e^{-nh(G_C)}$, which is less than or equal to $1$ for all three possible types of $G$ listed above. 

%It is simple to check that then the set $E(\mathcal{G})$ of labels of infinite paths in $\mathcal{G}$ is exactly $B$, the set of biinfinite concatenations of code words.

Then the map $\ell_C$ associates a label sequence to any path in $E(G_C)$; it is simple to check that any such label sequence is in $X_C$, and in fact in $B_C$, the set of biinfinite concatenations of code words. This suggests that $\ell_C$ may yield a correspondence between properties of $X_C$ (such as entropy) and of $E(G_C)$. There are, however, some immediate obstacles. First, $\ell_C$ is injective only when $C$ has unique decipherability, and so generally, $\ell_C$ may not preserve entropy. Secondly, the subset $B_C$ may have significantly lower entropy than $X_C$ itself. This in some sense reflects the intrinsic difference between the compact space $X$ and the noncompact $\ell(E(G_C)) = B_C$.

\begin{remark} 
One can of course try to deal with this issue by compactifying $G_C$, and in some works, including \cite{ruette}, results are shown regarding the one-point compactification $\overline{E(G)}$. However, this contains far less information than $X$ since different limits of pieces of loops have decidedly different limits in $L_C$, whereas all converge to the same point in $\overline{E(G_C)}$.
\end{remark}

%Theorem \ref{mainthm2} in some sense demonstrates that when $f(h(L)) < 1$, $\mathcal{G}$ does not in fact accurately reflect behavior of $X$, and Theorem~\ref{mainthm1} shows that that when $f(h(L)) = 1$, $\mathcal{G}$ may or may not contain behavior as complicated as that of $X$. 

As mentioned in the introduction, some aspects of our results are immediate corollaries of this connection to $E(G_C)$. For instance, here is a brief informal alternate proof of the portion of Theorem~\ref{mainthm3} implying $h(X_C) > h(L_C)$ (but not the formula for $h(X_C)$.) 
Assume that $f_C(h(L_C)) > 1$. Recall that $f_C(h(G_C)) \leq 1$ (since $|C_n| = s(u,n)$); since $f$ is decreasing, $h(G_C) > h(L_C)$. We also recall that it is always true that $h(X_C) \geq h(G_C)$, so $h(X_C) \geq h(G_C) > h(L_C)$. 

We will not need $E(G_C)$ again until Section~\ref{examples}, where we will use $\ell_C$ to give examples with and without measures of maximal entropy with support not contained in $L$, completing the proof of Theorem~\ref{mainthm1}. 

%Then, since $L_C$ is a shift-invariant subset of $X_C$, any ergodic measure of maximal entropy $\mu$ on $X_C$ has $\mu(L_C) = 0$. By Lemma~\ref{LB} (see Section~\ref{proofs}), $\mu((L_C \cup B_C)^c) = 0$. So, $\mu(B_C) = 1$, and so by unique decomposition (WRONG DEFINITION HERE) of $C$, $\ell$ is a measure-theoretic conjugacy between $(X_C, \mu)$ and $(E(G_C), \nu)$ for $\nu$ the push-forward of $\mu$ under $\ell^{-1}$. In particular, then $h(\nu) = h(\mu) = h(X_C) \geq h(G_C)$, and so $\nu$ is a measure of maximal entropy on $E(G_C)$ and $h(X_C) = h(G_C)$.

%Also, by Corollary~\ref{subwords}, 
%\[
%\limsup_{n \rightarrow \infty} \frac{\log |C_n|}{n} \leq h(L).
%\] 
%Therefore, since $h(L_C) < h(X_C) = h(G_C)$ in this case, by Corollary~\ref{subwords}, $|C_n| e^{-nh(G_C)}$ decays exponentially, and so automatically $\sum_{n = 1}^{\infty} n |C_n| e^{-nh(G_C)} < \infty$. This means that $G_C$ is strongly recurrent, and so $E(G_C)$ has a unique measure $\nu$ of maximal entropy. This means that $\mu$, the MME on $X_C$ chosen above, was in fact uniquely determined as the push-forward of $\nu$ under $\ell$.

\begin{remark}
We should note that extremely careful treatment of the behavior of $\ell_C$ under unique decipherability is exactly what's used in \cite{vaughntowers} to prove Theorem~\ref{vaughnthm} above, under the assumption that $h(L_C) < h(X_C)$.
\end{remark}

\begin{remark}
%Despite the proof just outlined, we present a self-contained proof of Theorem~\ref{mainthm3} in Section~\ref{proofs} due to the simplicity of the counting argument. 
Under the additional assumption of unique decipherability/decomposition on $C$, our Theorems~\ref{mainthm2} and \ref{mainthm1} could possibly alternately be proved by using the map $\ell_C$ and existing knowledge about transient/null recurrent countable graphs.
\end{remark}

\begin{remark}
In \cite{cec}, some related results were obtained, including an example of a coded subshift 
$X_C$ with $h(X_C) > h(G_C)$ and a method for computation of entropy of an arbitrary irreducible SFT by writing it as a coded shift with $|C| < \infty$ (corresponding to the case $L = \varnothing$) and solving $f_C(h(X_C)) = 1$, referred to as the ``loop method.'' He also proved the formula $f_C(h(X_C)) = 1$ for some examples with $|C| = \infty$, though not in complete generality. (See Example~\ref{petex} for more details.) 
\end{remark}

In the future, where a set $C$ of code words is fixed and there is no danger of ambiguity, we will omit the subscripts on $f_C$, $L_C$, $B_C$, $\ell_C$, and $X_C$ (and other auxiliary objects dependent on $C$) for readability.

\section{Proofs}\label{proofs}

We first verify the following simple claim from the introduction.

\begin{lemma}\label{LB}
For any $C$ and any measure $\mu$ on $X$, $\mu((L \cup B)^c) = 0$.
\end{lemma}

\begin{proof}
By definition, for any $x \in X \setminus B = \overline{B} \setminus B$, there exists $N$ so that either $x([N,n])$ is a subword of a word in $C$ for all $n > N$ or $x([n,N])$ is a subword of a word in $C$ for all $n < N$. If in addition $x \notin L$, then in the first case, there must exist a minimal such $N$, else $x \in L$; define $Y_N$ to be the set of all points associated to a minimal such $N$ in this way. Similarly, define $Z_N$ to be the set of all points in the second case associated to a maximal such $N$. Then, 
\[
(L \cup B)^c = \bigcup_{N \in \mathbb{Z}} Y_N \cup \bigcup_{N \in \mathbb{Z}} Z_N.
\]
However, clearly $Y_N = \sigma^n Y_0$ and $Z_N = \sigma^n Z_0$ for all $N$, and just as clearly, the sets $Y_N$ are all disjoint and the sets $Z_N$ are all disjoint. Therefore, they all must have zero measure for any measure $\mu$ on $X$, and so by countable additivity $(L \cup B)^c$ does as well.
\end{proof}

We also need a simple result relating $C_n$ and the language of the subshift $L$, which is nontrivial since only words appearing within words in $C$ arbitrarily far from the center are actually in the language of $L$. We can, however, use the following fact from \cite{vaughnronnie}.

\begin{lemma}{\rm (\cite{vaughnronnie}, Lemma 2.7)}
For any set of words $S$ which is closed under subwords, if we define $S_n = S \cap A^n$ and define the subshift
\[
Y(S) = \{y \in A^{\mathbb{Z}} \ : \ \forall k, y([-k,k]) \textrm{ is a subword of some word in } S\},
\] 
then $\displaystyle \limsup_{n \rightarrow \infty} \frac{\ln{|S_n|}}{n} = h(Y(S))$.
\end{lemma}

For any set $C$ of code words, if one defines $S$ to be the set of all subwords of $C$, then $S$ is closed under subwords and it's easily checked that $Y(S) = L_C$, yielding the following immediate corollary.

\begin{corollary}\label{subwords}
For a set $C$, define $W_n = W_n(C)$ to be the set of $n$-letter subwords of some word in $C$. Then
\[
\limsup_{n \rightarrow \infty} \frac{\ln{|W_n|}}{n} = h(L_C).
\]
\end{corollary}

We begin with the proofs of the auxiliary Theorems~\ref{aux1} and \ref{aux2}.

\begin{proof}[Proof of Theorem~\ref{aux1}]
Fix a set $C$ of code words, and suppose that $\alpha > h(L)$ and that $f(\alpha) < 1$. For future reference, denote by $P_n$, $S_n$, and $W_n$ the sets of $n$-letter prefixes, suffixes, and subwords 
(respectively) of words in $C$. (We again suppress the dependence on $C$ of these objects for readability.) Note that $P_n, S_n \subseteq W_n$. Then by Corollary~\ref{subwords}, for every $\epsilon$, there exists $M$ so that for all $n$,
\begin{equation}\label{wordbds}
|P_n|, |S_n|, |W_n| < M e^{n(h(L) + \epsilon)}.
\end{equation}

Our proof proceeds via simply bounding $|\mathcal{L}_n(X)|$ from above for each $n$. By the definition of $X$, every word in $\mathcal{L}(X)$ is a subword of a finite concatenation of words from $C$. We can partition words in $\mathcal{L}(X)$ by associating to any $w$ the smallest number $k$ of words in $C$ which must be concatenated to create a word containing $w$. This gives
\[
\mathcal{L}_n(X) = W_n \cup \left(\bigcup_{k=2}^\infty \bigcup_{n_1 + \ldots + n_k = n, n_i > 0} S_{n_1} C_{n_2} \ldots C_{n_{k-1}} P_{n_k}\right).
\]
%where juxtaposition of sets of words refers to all possible concatenations (in the same order) of words from the sets. 
%We note that the concatenations inside the union are nonempty for only finitely many values of $k$; for instance $k > n$ is impossible. 
This yields the following inequality:
\[
|\mathcal{L}_n(X)| \leq |W_n| + \sum_{k=2}^{\infty} \sum_{n_1 + \ldots + n_k = n, n_i > 0} |S_{n_1}| |C_{n_2}| \ldots |C_{n_{k-1}}| |P_{n_k}|.
\]
Choose $\epsilon = \alpha - h(L)$ and apply (\ref{wordbds}):
\[
|\mathcal{L}_n(X)| < Me^{n\alpha} + \sum_{k=2}^{\infty} \sum_{n_1 + \ldots + n_k = n, n_i > 0} M^2 e^{(n_1 + n_k)\alpha} \prod_{i=2}^{k-1} |C_{n_i}|.
\]
Some factoring yields
\[
|\mathcal{L}_n(X)| < e^{n\alpha} \left( M + n M^2 \sum_{k=2}^{\infty} \sum_{n_2 + \ldots + n_{k-1} = n, n_i > 0} \prod_{i=2}^{k-1} |C_{n_i}| e^{-n_i \alpha}\right).
\]
(The extra factor of $n$ appears because a particular choice for $n_2, \ldots, n_{k-1}$ could correspond to several different choices of 
$n_1, \ldots, n_k$, but not more than $n$.) Finally, we note that all terms in the second sum are part of the expansion of $(f(\alpha))^{k-2} = \left( \sum_{j = 1}^{\infty} |C_j| e^{-j\alpha} \right)^{k-2}$, and so
\[
|\mathcal{L}_n(X)| < e^{n\alpha} \left( M + nM^2 \sum_{k=2}^{\infty} (f(\alpha))^{k-2} \right).
\]
Since $f(\alpha) < 1$, we rewrite as
\[
|\mathcal{L}_n(X)| < e^{n\alpha} \left(M + nM^2 \frac{1}{1 - f(\alpha)}\right).
\]
Taking logarithms, dividing by $n$, and letting $n \rightarrow \infty$ shows $h(X) \leq \alpha$.

\end{proof}

\begin{proof}[Proof of Theorem~\ref{aux2}]

Suppose that $C$ has unique decomposition and that $f(\alpha) > 1$. Then clearly we can choose $t$ so that $\sum_{j=1}^t |C_j| e^{-j\alpha} > 1$; denote this truncated sum by $\eta$. For each $k$, consider the expansion 
\[
\eta^k = \left( \sum_{j = 1}^{t} |C_j| e^{-j\alpha} \right)^k = \sum_{n = k}^{tk} \sum_{n_1 + \ldots + n_k = n, 0 < n_i \leq t} e^{-n\alpha} \prod_{i=1}^k 
|C_{n_i}|.
\]

Clearly we can choose $n = N_k$ for which the first sum is maximized, yielding
\[
\frac{\eta^k}{tk} < \sum_{n_1 + \ldots + n_k = N, 0 < n_i \leq t} e^{-N\alpha} \prod_{i=1}^k |C_{n_i}| \Rightarrow 
\sum_{n_1 + \ldots + n_k = N, 0 < n_i \leq t} \prod_{i=1}^k |C_{n_i}| > \frac{e^{N\alpha} \eta^k}{tk}.
\]

We note that for every choice of $n_1, \ldots, n_k > 0$ with $\sum n_i = N$, the sets of concatenations $C_{n_1} \ldots C_{n_k}$ are all in $\mathcal{L}_N(X)$, and that by unique decomposition, for different $k$-tuples $(n_i) \neq (n'_i)$, the associated collections of words are disjoint. Therefore,
\[
|\mathcal{L}_N(X)| \geq \sum_{n_1 + \ldots + n_k = N, 0 < n_i \leq t} \prod_{i=1}^k |C_{n_i}| > \frac{e^{N\alpha} \eta^k}{tk}.
\]
Recall that $k \leq N \leq tk$, and so
\[
|\mathcal{L}_N(X)| > \frac{e^{N\alpha} \eta^{N/t}}{tN}.
\]
Since $N \geq k$, taking $k$ to infinity will force $N$ to approach infinity. We can then take logarithms, divide by $N$, and let $N$ approach infinity to get
\[
h(X) \geq \alpha + t^{-2} \ln \eta.
\]
Since $\eta > 1$, this shows that $h(X) > \alpha$, completing the proof.

\end{proof}

We can now present the proofs of Theorems~\ref{mainthm1} and \ref{mainthm3}.

\begin{proof}[Proof of Theorem~\ref{mainthm1}]
Suppose that $f(h(L)) \leq 1$. Then, for every $\alpha > h(L)$, $f(\alpha) < 1$ since $f$ is strictly decreasing. By Theorem~\ref{aux1}, then $h(X) \leq \alpha$. Since $\alpha > h(L)$ was arbitrary, $h(X) \leq h(L)$. However, since $L \subseteq X$, $h(L) \leq h(X)$ trivially, and so $h(L) = h(X)$. Examples~\ref{nomoreMME} and \ref{moreMME} from Section~\ref{examples} will demonstrate that $X$ can either possess or not possess a measure of maximal entropy with support not contained in $L$. 
\end{proof}

\begin{proof}[Proof of Theorem~\ref{mainthm3}]
Suppose that $C$ has unique decomposition and that $f(h(L)) > 1$. Since $|C_n| < |A|^n$ for every $n$, $\lim_{x \rightarrow \infty} f(x) = 0$. Therefore, by the Intermediate Value Theorem, there exists $x > h(L)$ for which $f(x) = 1$. We need to show that $h(X) = x$.

Since $f$ is strictly decreasing, for every $\alpha > x$, $f(\alpha) < 1$, and so $h(X) \leq \alpha$ by Theorem~\ref{aux1}. Since $\alpha > x$ was arbitrary, $h(X) \leq x$. Similarly, for every $\alpha \in [h(L), x)$, $f(\alpha) > 1$, and so $h(X) \geq \alpha$ by Theorem~\ref{aux2}; again since $\alpha \in [h(L), x)$ was arbitrary, $h(X) \geq x$, completing the proof that $h(X) = x$.

%Finally, we must show that $X$ has a unique measure of maximal entropy. Though the connection from Section~\ref{connections} could yield a proof, for simplicity we instead cite the following theorem from (VAUGHN-DAN REFERENCE).

%\begin{theorem}{\rm(REFERENCE, Theorem B)}\label{CTthm}
%If $X$ is a coded subshift and $\limsup \frac{\log |W_n|}{n} < h(X)$, then $X$ is intrinsically ergodic.
%\end{theorem}

%By Corollary~\ref{subwords}, $\limsup \frac{\log |W_n|}{n} = h(L) < h(X)$, and so Theorem B can be applied, completing the proof.

\end{proof}

Finally we must prove Theorem~\ref{mainthm2}, which requires a slightly different counting argument.

\begin{proof}[Proof of Theorem~\ref{mainthm2}]
Suppose that $f(h(L)) < 1$, and consider any word $u \in \mathcal{L}(X) \setminus \mathcal{L}(L)$. Since $u \notin \mathcal{L}(L)$, there exists $N$ so that $u$ does not appear as a subword of any $C$-word at a location with distance more than $N$ from the beginning and end. (If this were not the case, then $u$ would be contained in a sequence of $C$-words at distances arbitrarily far from the ends, implying $u \in \mathcal{L}(L)$.) 

Choose any ergodic measure $\mu$ with $\mu([u]) > 0$. Then, for every $n$, define 
\[
G_n = \{w \in \mathcal{L}_n(X) \ : \ w \textrm{ contains at least } n (\mu([u])/2) \textrm{ occurrences of } u\}.
\]
By the ergodic theorem, $\mu(G_n) \rightarrow 1$, and so by standard arguments using definition of entropy (for a formal proof, see for example Lemma 4.8 of \cite{perturb},  
\begin{equation}\label{generic}
\liminf_{n \rightarrow \infty} \frac{\ln |G_n|}{n} \geq h(\mu). 
\end{equation}
For any $w \in G_n$, we may decompose it as $s_1 w_2 \ldots w_{k-1} p_k$, where $s_1$ is a suffix of a word in $C$, $p_k$ is a prefix of a word in $C$, and each $w_i$ is in $C$. By definition of $G_n$, $w$ contains at least $n (\mu([u])/2)$ occurrences of $u$, each of which either contains or is within distance $N$ of one or more of: the beginning of $w$, the end of $w$, or one of the $k-1$ transitions in the concatenation $s_1 w_2 \ldots w_{k-1} p_k$. This clearly implies that $n (\mu([u])/2) \leq 2N + (k-1)(|u| + 2N) \leq k(|u| + 2N)$. Put another way, every word in $G_n$ has a decomposition as a subword of a concatenation of $k$ words from $C$, where $k \geq \beta n$ for $\beta = \mu([u])/{2(|u| + 2N)} > 0$.  

Define $P_n$, $S_n$, and $W_n$ as in the proof of Theorem~\ref{aux1}. Then a similar counting argument to the one used there yields
\[
|G_n| \leq \sum_{k \geq \beta n} \sum_{n_1 + \ldots + n_k = n, n_i > 0} |S_{n_1}| |C_{n_2}| \ldots |C_{n_{k-1}}| |P_{n_k}|.
\]
Choose any $0 < \epsilon < -\beta \ln f(h(L))$ and apply (\ref{wordbds}) to get
\[
|G_n| \leq \sum_{k \geq \beta n} n M^2 e^{n(h(L) + \epsilon)} \sum_{n_2 + \ldots + n_{k-1} = n, n_i > 0} \prod_{i=2}^{k-1} |C_{n_i}| e^{-n_i (h(L) + \epsilon)}.
\]
Just as in the proof of Theorem~\ref{aux1}, the inner sum is less than $f(h(L) + \epsilon)^{k-2}$, which in turn is less than 
$(f(h(L)))^{k-2}$, so
\[
|G_n| \leq nM^2 e^{n(h(L) + \epsilon)} \sum_{k \geq \beta n} (f(h(L)))^{k-2}.
\]
Since $f(h(L)) < 1$, we rewrite as
\[
|G_n| \leq \frac{nM^2}{1 - f(h(L))} e^{n(h(L) + \epsilon)} (f(h(L)))^{\beta n}.
\]
Taking logarithms, dividing by $n$, and taking the limit infimum as $n$ approaches infinity (and recalling (\ref{generic})) yields
\[
h(\mu) \leq h(L) + \epsilon + \beta \ln f(h(L)) < h(L).
\]
Since $h(X) \geq h(L)$ (in fact $h(X) = h(L)$ by Theorem~\ref{mainthm1}), $\mu$ is not a measure of maximal entropy. Since $\mu$ was an arbitrary ergodic measure giving $[u]$ positive measure, we know that all ergodic measures of maximal entropy for $X$ give $[u]$ zero measure. Since every measure of maximal entropy on a subshift is a convex combination of ergodic measures of maximal entropy, in fact all measures of maximal entropy give $[u]$ zero measure. Finally, since $u \in \mathcal{L}(X) \setminus \mathcal{L}(L)$ was arbitrary, we are done.

\end{proof}

\section{examples}\label{examples}

We first give examples of $C$ with $f_C(h(L_C)) = 1$ and where $X_C$ either has or does not have an MME with support not contained in $L_C$, completing the proof of Theorem~\ref{mainthm1}. For both of these examples, we use the associated countable state Markov chain $E(G_C)$ defined in 
Section~\ref{markov}. In general, the lack of structure of $L_C$ (for instance, it could be the case that $L_C = B_C = X_C$!) makes the analysis of the $f_C(h(L_C)) = 1$ case via $E(G_C)$ intractable. However, throughout this section all examples will have unique decipherability and will satisfy 
$B_C \cap L_C = \varnothing$, for which we have the following useful result. 

\begin{proposition}\label{extraMME}
If $C$ has unique decipherability, $B_C \cap L_C = \varnothing$, and $f_C(h(L_C)) = 1$, then $X_C$ has an ergodic measure of maximal entropy with support not contained in $L_C$ if and only if the associated countable state Markov chain $E(G_C)$ has a measure of maximal entropy.
\end{proposition}

\begin{proof}
As usual, we suppress dependence on $C$ for $X, f, L, B, G$, and $\ell$ throughout the proof. Recall from Section~\ref{markov} that the inequalities $f(h(G)) \leq 1$ and $h(X) \geq h(G)$ always hold. Since $f(h(L)) = 1$, this means that $f(h(G)) \leq f(h(L))$ and so that $h(X) \geq h(G) \geq h(L)$. However, since $f(h(L)) = 1$, Theorem~\ref{mainthm1} implies that $h(X) = h(L)$, and so that all inequalities above are equalities. If $\mu$ is an ergodic measure on $X$, then by Lemma~\ref{LB} and the disjointness of $B$ and $L$, $\mu$ has support not contained in $L$ if and only if $\mu(B) = 1$. 

Recall the label map $\ell$ from Section~\ref{markov}, which maps $E(G)$ onto $B$. Since $C$ has unique decipherability, $\ell$ is bijective, and since it is also a Borel map which commutes with the shift, it yields a bijection between the measures of maximal entropy $h(X)$ on $X$ which give $B$ measure $1$ and the measures of maximal entropy $h(G) = h(X)$ on $E(G)$. Therefore, $E(G)$ has a measure of maximal entropy iff $X$ has an ergodic measure $\mu$ with measure not contained in $L$ and entropy $h(X)$. 

By the ergodic decomposition, $X$ has an ergodic measure with support not contained in $L$ and maximal entropy $h(X)$ iff it has any measure (not necessarily ergodic) with these properties, and so the proof is complete.

\end{proof}

\begin{remark}

We suspect that with more work, Proposition~\ref{extraMME} might be provable under the weaker assumption of unique decomposition. However, all examples to which we apply Proposition~\ref{extraMME} will satisfy unique decipherability anyway, and so we do not need any such extensions in this work.

\end{remark}

We may now present the examples for Theorem~\ref{mainthm1}.

\begin{example}\label{nomoreMME}
Define $C_{2n} = \{a_1 \ldots a_n 0^n \ : \ a_i \in \{1,2\}\}$ for all $n \in \mathbb{N}$, $C_j = \varnothing$ for all odd $j$, and $C = \bigcup_{j \in \mathbb{N}} C_j$. Then $C$ clearly has unique decipherability, $L = \{1,2\}^{\mathbb{Z}} \cup \{0^{\infty}\} \cup \{x \ : \ \exists k \textrm{ s.t. } \forall i < k \ x_i \in \{1,2\}, \forall i \geq k \ x_i = 0\}$, and so $h(L) = \ln 2$. Then,
\[
f(h(L)) = \sum_{j = 1}^{\infty} |C_j| e^{-jh(L)} = \sum_{n = 1}^{\infty} 2^n 2^{-2n} = 1.
\]
As in the proof of Proposition~\ref{extraMME}, we now know that $h(X) = h(G) = h(L) = \ln 2$, and so
\[
\sum_{n = 1}^{\infty} n s(n,u) e^{-nh(G)} = \sum_{n = 1}^{\infty} 2n 2^{n} 2^{-2n} = 4 < \infty.
\]
Therefore, $E(G)$ is positive recurrent, and so has a measure with entropy $h(G)$. By Proposition~\ref{extraMME}, $X$ has an MME with support not contained in $L$.
\end{example}

\begin{example}\label{moreMME}
Define $C_{n + \lfloor \log_2 n \rfloor} = \{a_1 \ldots a_n 0^{\lfloor \log_2 n \rfloor} \ : n \geq 2, \ a_i \in \{1,2,3,4\}\}$ for every 
$n \in \mathbb{N}$, $C_j = \varnothing$ for all other $j \in \mathbb{N}$, and $C = \bigcup_{j \in \mathbb{N}} C_j$. Then $C$ clearly has unique decipherability, 
$L = \{1,2,3,4\}^{\mathbb{Z}} \cup \{0^{\infty}\} \cup \{x \ : \ \exists k \textrm{ s.t. } \forall i < k \ x_i \in \{1,2,3,4\}, \forall i \geq k \ x_i = 0\}$, and so $h(L) = \ln 4$. Then,
\[
f(h(L)) = \sum_{j = 1}^{\infty} |C_j| e^{-jh(L)} = \sum_{n = 2}^{\infty} 4^n 4^{-(n + \lfloor \log_2 n \rfloor)} = 1.
\]
As in the proof of Proposition~\ref{extraMME}, we now know that $h(X) = h(G) = h(L) = \ln 4$, and so
\[
\sum_{n = 1}^{\infty} n s(n,u) e^{-nh(G)} = \sum_{n = 2}^{\infty} (n + \lfloor \log_2 n \rfloor) 4^n 4^{-(n + \lfloor \log_2 n \rfloor)} \geq \sum_{n = 2}^{\infty} n 4^{-\log_2 n} = \sum_{n = 2}^{\infty} \frac{1}{n} = \infty.
\]
Therefore, $E(G)$ is null recurrent, and so does not have a measure with entropy $h(G)$. By Proposition~\ref{extraMME}, $X$ does not have an MME with support not contained in $L$.
\end{example}

We conclude with a few examples from the literature which can be treated via our more general results. 
First, we consider the Dyck shift of \cite{krieger}.

\begin{example}
Choose the alphabet $A = \{( \ , \ [ \ , \ ) \ , \ ]\}$, and define $C$ to be the set of all minimal words which reduce to the identity under the relations $( \ ) = [ \ ] = \textrm{id}$. For example, $( \ [ \ ] \ ) \in C$, but $( \ [ \ ) \ ] \notin C$ since it does not reduce to the identity, and $( \ ) \ [ \ ] \notin C$ since it does reduce to the identity, but is not minimal with that property, since it can be decomposed into $( \ )$ and $[ \ ]$, which both reduce to the identity (and are in $C$.)

Then $X = X_C$ is the so-called Dyck shift. Also, it is easily checked that every proper prefix of a word in $C$ must have strictly more left parentheses/brackets than right ones, and the opposite is true for a proper suffix of a word in $C$. Therefore, no word can be both a proper prefix and proper suffix of words in $C$, and so $C$ has unique decipherability. Also, $L = L_C = X$; this is most easily seen by noting that for any $w_1, \ldots, w_k \in C$, the word $( w_1 \ldots w_k )$ (where the $w_i$ are concatenated and then surrounded by one set of parentheses) is in $C$. Finally, it is known that $h(X) = \ln 3$ and that $X$ has multiple measures of maximal entropy (in fact exactly two ergodic ones; see \cite{krieger}). 

Therefore, by Corollary~\ref{maincor}, it should be the case that $f(h(L)) \leq 1$, and in fact this is the case. It's well known that the number of ways to arrange $n+1$ sets of matching parentheses in an indecomposable way (without the two types listed above) is the $n$th Catalan number $\frac{(2n-2)!}{n!(n-1)!}$. For any $n \geq 1$, $C_{2n}$ is obtained by labeling each set of parentheses in every such word by one of two types (parentheses or brackets), and so $|C_{2n}| = \frac{(2n-2)!}{n!(n-1)!} 2^{n}$. 
Also, the generating function for the Catalan numbers is
\[
\sum_{n = 1}^{\infty} \frac{(2n-2)!}{n!(n-1)!} x^n = \frac{1 - \sqrt{1 - 4x}}{2}.
\]
Therefore, $f(h(L)) = f(\ln 3) = $
\[
\sum_{n = 1}^{\infty} \frac{(2n-2)!}{n!(n-1)!} 2^n e^{-2n \ln 3}
= \sum_{n = 1}^{\infty} \frac{(2n-2)!}{n!(n-1)!} \left(\frac{2}{9} \right)^n = 
\frac{1 - \sqrt{1 - 4(2/9)}}{2} = \frac{1}{3} \leq 1,
\]
as expected.
\end{example}

We now give some examples from \cite{pavlovspec}, which are coded subshifts with various weakened specification properties and multiple measures of maximal entropy. 

\begin{example} 
Choose $N > e^6$, take $A = \{-N, \ldots, N\}$, and define $C = \{0, w_1 \ldots w_n 0^k \ : \ \textrm{all } w_i \neq 0, \textrm{ all } w_i \textrm{ have the same sign}, k = 1 + \lfloor \ln n \rfloor\}$. 
It is shown in \cite{pavlovspec} that the induced coded subshift $X$ has a property called non-uniform specification (with gap function $1 + \lfloor \ln n \rfloor$).

It is easy to check that $L = \{-N,\ldots,-1\}^{\mathbb{Z}} \cup \{1,\ldots,N\}^{\mathbb{Z}} \cup \{0^{\infty}\} \cup \{x \ : \ \exists k \textrm{ s.t. } \forall i < k \ x_i < 0, \forall i \geq k \ x_i = 0\} \cup \{x \ : \ \exists k \textrm{ s.t. } \forall i < k \ x_i > 0, \forall i \geq k \ x_i = 0\}$, and so $h(L) = \ln N$. Then $f(h(L)) = $
\begin{multline*}
\sum_{n = 1}^{\infty} |C_n| e^{-n h(L)} = e^{-\ln N} + \sum_{n = 1}^{\infty} 2N^n e^{-(n + 1 + \lfloor \ln n \rfloor) \ln N}
= \frac{1}{N} \left(1 + 2 \sum_{n=1}^{\infty} \frac{1}{N^{\lfloor \ln n \rfloor}} \right)\\
\leq \frac{3}{N} + 2 \sum_{n=2}^{\infty} \frac{1}{N^{\ln n}} = \frac{3}{N} + 2 \sum_{n=2}^{\infty} \frac{1}{n^{\ln N}} \leq \frac{3}{N} + 2\int_2^{\infty} x^{-\ln N} = \frac{3}{N} + \frac{2}{\ln N - 1},
\end{multline*}
which is less than $1$ since $N > e^6$. Then by Theorem~\ref{mainthm2}, $h(X) = h(L) = \ln N$, and so $X$ has multiple measures of maximal entropy, namely the uniform Bernoulli measures on the disjoint full shifts $\{-N,\ldots,-1\}^{\mathbb{Z}}$ and $\{1,\ldots,N\}^{\mathbb{Z}}$, which are both contained in $L$.
\end{example}

\begin{example}
Choose any $N \geq 10$ and define $A = \{-N, \ldots, -1, 1, \ldots, N\}$. For every $n$, define $P_n$ to be a subset of $\{1,\ldots,N\}^n$ of minimal size which is $2$-spanning in the Hamming metric, and define $N_n = -P_n$. Define $C = \{uv \ : \ \exists n,m \textrm{ s.t. } u \in P_n, v \in N_m\}$. 
It is shown in \cite{pavlovspec} that the induced coded subshift $X$ has a property called almost specification (with gap function 
$g(n) = 4$).

The reader may check that
$L = \{-N,\ldots,-1\}^{\mathbb{Z}} \cup \{1,\ldots,N\}^{\mathbb{Z}} \cup \{x \ : \ \exists k \textrm{ s.t. } \forall i < k \ x_i < 0, \forall i \geq k \ x_i > 0\} \cup \{x \ : \ \exists k \textrm{ s.t. } \forall i < k \ x_i > 0, \forall i \geq k \ x_i < 0\}$, and so 
$h(L) = \ln N$. 

To bound $f(h(L))$, we need bounds on $|P_n| = |N_n|$ for $n \in \mathbb{N}$. Clearly $|P_1| = |P_2| = 1$, and it is shown in \cite{pavlovspec} that for $n > 2$, $|P_n| \leq \min(N^{n-2}, 16n^{-2} N^n)$. Therefore, 
\begin{multline*}
f(h(L)) = \sum_{n = 1}^{\infty} |C_n| e^{-n h(L)} = \sum_{n = 1}^{\infty} \sum_{i=1}^{n-1} |P_i| |N_{n-i}| N^{-n} 
= \left( \sum_{m = 1}^{\infty} |P_m| N^{-m} \right)^2 \\
\leq \left(N^{-1} + \sum_{m = 2}^{\infty} \min(N^{-2}, 16m^{-2}) \right)^2 \leq \left(N^{-1} + 31N^{-2} + 16 \sum_{m=33}^{\infty} m^{-2}\right)^2 \\
\leq (0.1 + 0.31 + 16 \int_{32}^{\infty} x^{-2})^2 = (0.41 + 0.5)^2 < 1.
\end{multline*}
Then by Theorem~\ref{mainthm2}. $h(X) = h(L) = \ln N$, and again $X$ has multiple measures of maximal entropy given by the uniform Bernoulli measures on the disjoint full shifts $\{-N,\ldots,-1\}^{\mathbb{Z}}$ and $\{1,\ldots,N\}^{\mathbb{Z}}$, both of which are contained in $L$.
\end{example}

\begin{remark}
We do not give a full description here as the examples are a little more technical, but \cite{kwietniaketal} contains different subshifts with the same weakened specification properties and multiple MMEs, which are also coded systems which could be placed into our framework. 
\end{remark}

\begin{example}\label{petex}
In \cite{cec}, Petersen describes the so-called loop method for computing the entropy of an irreducible nearest-neighbor SFT $Y$. For any such $Y$, choose a distinguished letter $a$ and for $i \in \mathbb{N}$, define $T_i$ to be the set of words of length $i + 2$ which begin and end with $a$ and do not contain any other $a$. If we then define $C = \{w \ : \ \exists i \textrm{ s.t. }
wa \in T_i\}$, then it's easily checked that $C$ has unique decipherability and that $Y = X_C$, the coded subshift induced by $C$. By irreducibility of $Y$, it's simple to see that $L$ is just $Y^{(a)}$, the subshift of $Y$ consisting of points which do not contain $a$.

Then, by irreducibility it's simple to show that there exists a distance $D$ so that $|T_i| \geq |\mathcal{L}_{i-D}(Y^{(a)})|$ for all $a$, and by (\ref{wordcount}), this is greater than or equal to $e^{h(Y^{(a)}) (i-D)}$. Therefore,
\[
f(h(L)) = \sum_{n=1}^{\infty} |C_n| e^{-nh(L)} = \sum_{n=1}^{\infty} |T_{n+1}| e^{-nh(Y^{(a)})} \geq \sum_{n=1}^{\infty} e^{-h(Y^{(a)})(d-1)} = \infty > 1.
\]

Then by Theorem~\ref{mainthm3}, $h(Y)$ is the unique root $\alpha$ of 
\[
\sum_{n=1}^{\infty} |C_n| e^{-n\alpha} = 1.
\]
Since $|C_n| = |T_{n-1}|$ for $n > 0$, this is the same as the logarithm of the root $x$ of
\begin{equation}\label{loop}
\sum_{i=0}^{\infty} \frac{|T_i|}{x^{i+1}} = 1,
\end{equation}
which is precisely the formula described in \cite{cec}.

\end{example}

\begin{remark}

As long as $Y$ is an irreducible subshift for which $a$ is a synchronizing letter, it is true that 
$Y = X_C$ for $C$ as above. The reader may check that again $L = Y^{(a)}$. Therefore, whenever $Y$ is an irreducible subshift for which $a$ is a synchronizing letter and 
\[
\sum_{n=1}^{\infty} |T_{n+1}| e^{-nh(Y^{(a)})} > 1
\]
(in particular, if $h(Y) > h(Y^{(a)})$), then the same formula for $h(Y)$ holds, i.e. that it is the logarithm of the root $x$ of (\ref{loop}).

\end{remark}

\bibliographystyle{plain}
\bibliography{coded}

\begin{thebibliography}{10}

\bibitem{codedbook}
Jean Berstel and Dominique Perrin.
\newblock {\em Theory of codes}, volume 117 of {\em Pure and Applied
  Mathematics}.
\newblock Academic Press, Inc., Orlando, FL, 1985.

\bibitem{blanhan}
F.~Blanchard and G.~Hansel.
\newblock Syst\`emes cod\'es.
\newblock {\em Theoret. Comput. Sci.}, 44(1):17--49, 1986.

\bibitem{vaughntowers}
Vaughn Climenhaga.
\newblock {Specification and towers in shift spaces}.
\newblock {\em ArXiv e-prints}, February 2015.

\bibitem{vaughnronnie}
Vaughn Climenhaga and Ronnie Pavlov.
\newblock {One-sided almost specification and intrinsic ergodicity}.
\newblock {\em ArXiv e-prints}, October 2017.

\bibitem{CT0}
Vaughn Climenhaga and Daniel~J. Thompson.
\newblock Intrinsic ergodicity beyond specification: {$\beta$}-shifts,
  {$S$}-gap shifts, and their factors.
\newblock {\em Israel J. Math.}, 192(2):785--817, 2012.

\bibitem{kwietcoded}
Jeremias Epperlein, Dominik Kwietniak, and Piotr Oprocha.
\newblock {Mixing properties in coded systems}.
\newblock {\em ArXiv e-prints}, March 2015.

\bibitem{fiebigs}
Doris Fiebig and Ulf-Rainer Fiebig.
\newblock Covers for coded systems.
\newblock In {\em Symbolic dynamics and its applications ({N}ew {H}aven, {CT},
  1991)}, volume 135 of {\em Contemp. Math.}, pages 139--179. Amer. Math. Soc.,
  Providence, RI, 1992.

\bibitem{gurevic3}
B.~M. Gurevi{\v{c}}.
\newblock Shift entropy and {M}arkov measures in the space of paths of a
  countable graph.
\newblock {\em Dokl. Akad. Nauk SSSR}, 192:963--965, 1970.

\bibitem{krieger}
Wolfgang Krieger.
\newblock On the uniqueness of the equilibrium state.
\newblock {\em Math. Systems Theory}, 8(2):97--104, 1974/75.

\bibitem{kriegercoded}
Wolfgang Krieger.
\newblock On subshifts and topological {M}arkov chains.
\newblock In {\em Numbers, information and complexity ({B}ielefeld, 1998)},
  pages 453--472. Kluwer Acad. Publ., Boston, MA, 2000.

\bibitem{kwietniaketal}
Dominik Kwietniak, Piotr Oprocha, and Micha{\l} Rams.
\newblock On entropy of dynamical systems with almost specification.
\newblock {\em Israel J. Math.}, 213(1):475--503, 2016.

\bibitem{perturb}
Ronnie Pavlov.
\newblock Perturbations of multidimensional shifts of finite type.
\newblock {\em Ergodic Theory Dynam. Systems}, 31(2):483--526, 2011.

\bibitem{pavlovspec}
Ronnie Pavlov.
\newblock On intrinsic ergodicity and weakenings of the specification property.
\newblock {\em Adv. Math.}, 295:250--270, 2016.

\bibitem{cec}
Karl Petersen.
\newblock Chains, entropy, coding.
\newblock {\em Ergodic Theory Dynam. Systems}, 6(3):415--448, 1986.

\bibitem{ruette}
Sylvie Ruette.
\newblock On the {V}ere-{J}ones classification and existence of maximal
  measures for countable topological {M}arkov chains.
\newblock {\em Pacific J. Math.}, 209(2):366--380, 2003.

\bibitem{verejones}
D.~Vere-Jones.
\newblock Ergodic properties of nonnegative matrices. {I}.
\newblock {\em Pacific J. Math.}, 22:361--386, 1967.

\bibitem{walters}
Peter Walters.
\newblock {\em An Introduction to Ergodic Theory}.
\newblock Number~79 in Graduate Texts in Mathematics. Springer-Verlag, 1982.

\end{thebibliography}

\end{document}